\documentclass{amsart}
\usepackage{amscd} 
\usepackage{amsfonts} 
\usepackage{amssymb} 
\usepackage{latexsym} 
\usepackage{diagrams}
\newcommand{\ncm}{\newcommand}

\def\C{\mathbb{C}\,}

\newtheorem{theorem}{Theorem}[section]
\newtheorem{prop}[theorem]{Proposition}

\newtheorem{cor}[theorem]{Corollary}

\newtheorem{lem&def}[theorem]{Lemma \& Definition}

\newtheorem{exa}[theorem]{Example}

\newtheorem{remark}[theorem]{Remark}

\ncm{\End}{\mbox{\rm End}\,}

\def\Hom{\mbox{\rm Hom}\,}

\def\Im{\mbox{\rm Im}\,}

\def\id{\mbox{\rm id}}

\def\into{\hookrightarrow}
\def\to{\rightarrow}

\def\o{\otimes}    
\def\bra{\langle}
\def\ket{\rangle}

\ncm{\rarr}[1]{\stackrel{#1}{\longrightarrow}}
\ncm{\larr}[1]{\stackrel{#1}{\longleftarrow}}
\def\cop{\Delta}

\def\eps{\varepsilon}

\def\du1{\hat 1}

\def\-1{_{(-1)}}
\def\0{_{(0)}}
\def\1{_{(1)}}
\def\2{_{(2)}}
\def\3{_{(3)}}
\def\4{_{(4)}}

\def\|{\, | \,}

\def\du1{\hat 1}
\def\lact{\triangleright}

\begin{document}

\title{Semisimple Hopf algebras and their depth two Hopf subalgebras}
\author{Lars Kadison}
\address{Department of Mathematics \\ Louisiana State University \\ 
Baton Rouge, LA 70803 \\
and University of Pennsylvania \\
David Rittenhouse Labratory \\
209 South 33rd Street \\ Philadelphia, PA 19104-6395} 
\email{lkadison@math.lsu.edu lkadison@math.upenn.edu} 
\date{}
\thanks{}
\date{\today} 

\begin{abstract} 
We prove   
 that a depth two Hopf subalgebra $K$ of a semisimple
Hopf algebra $H$ is normal (where the ground field $k$ is
algebraically closed of characteristic zero). 
This means on the one hand that a Hopf subalgebra is normal
when inducing (then restricting) modules several times as opposed to one time creates no new simple constituents.  This point of view was taken
in the paper \cite{KK} which established a normality result in case $H$ and $K$ are finite group algebras.   On the other hand this means
that $K$ is normal in $H$ when $H \| K$ is a Galois extension with respect to action of generalized bialgebras such as bialgebroids, weak Hopf algebras or Hopf algebroids.
The generalized Galois picture of depth two is the point of view we take here:  after showing the
centralizer $R$ is separable algebra via Hopf invariant theory, we compute that the depth two semisimple Hopf algebra pair $H \| K$ is free Frobenius extension with Markov trace satisfying all hypotheses considered in \cite[Kadison and Nikshych, Frobenius Extensions and Weak Hopf Algebras]{KN}.
By the main theorem in that paper it is then a Galois extension with action of semisimple weak Hopf algebra (also regular and possessing Haar integral).  Then the Galois canonical isomorphism (via coring theory) restricted to integral induces algebra homomorphism from  Hopf
algebra into weak Hopf algebra with kernel $HK^+ = K^+H$.    
\end{abstract} 
\maketitle

\section{Introduction}

Finite depth originated as a notion in the classification of type
$II_1$ subfactors.  Ocneanu saw in the late eighties that especially
depth two has extraordinary algebraic properties, which he phrased in
terms of paragroups.  A realization of this project in algebra occured
in stages starting with Szymanski and others, and proceeding with a score of  papers, of which \cite{NV, KN, KS, Belg, ferrara, LK, anchor}
are somewhat representative.  Critical input in the shaping of this algebraic theory came from 
results in Hopf-Galois extensions in the eighties
and early nineties, the development of weak Hopf algebras and Hopf algebroids in the mid- to late nineties (and the change of definition of antipode coming from consideration of depth two Frobenius extension
by B\"ohm-Szlachanyi).  

The depth two condition for a subgroup $H$ of a finite group $G$ is
in terms of irreducible characters $\chi \in {\rm Irr}(G)$ and $\psi \in {\rm Irr}(H)$, that there is positive integer $n$ such that $\langle {\rm Ind}_H^G ({\rm Res}_H^G ({\rm Ind}_H^G (\psi))) \|  \chi \rangle_G \le n \langle {\rm Ind}_H^G (\psi) \| \chi \rangle_G$. In the paper \cite[author-K\"ulshammer]{KK} it is shown via the Mackey subgroup theorem that this condition entails that $H$ is a normal subgroup of $G$.  

For general subrings the depth two property splits into a left- and right-handed depth two conditions (for associative unital rings and unit-preserving ring homomorphisms).  A subring $B$ is right depth two
in ring $A$ if there is a split $A$-$B$-bimodule epimorphism from
some $A^n$ onto $A \otimes_B A$ \cite{KS}. Finitely generated Hopf-Galois extensions satisfy
a particularly strong form of this condition, where isomorphism replaces
split epi.  

 It is not hard to check that for an 
equivalent right depth two condition on a ring extension $A \| B$, we impose the existence of $n$ mappings $\gamma_j \in \End {}_BA_B$ and $n$ elements $u_j \in (A \otimes_B A)^B$ such that
$x \otimes_B y = \sum_{j=1}^n x \gamma_j(y) u_j$ for all $x,y \in A$.  For example,
a normal subgroup $N$ of index $n$ in group $G$ induces
depth 2 subalgebra $B = k[N]$ in group algebra $k[G]$ with $u_j = g_j^{-1} \otimes_B g_j$, $\gamma_j(\sum_{g \in G} n_g g) = \sum_{h \in H} n_{hg_j} hg_j$ where $g_1,\ldots,g_n$ are coset representatives of $N$ in $G$.

Right depth two subrings $B \into A$ have a Galois theory over bialgebroid structure on $\End {}_BA_B$ and a dual bialgebroid structure on $(A \otimes_BA)^B$. Bialgebroids enjoy axioms like a bialgebra except over
a noncommutative base ring, which in the case of depth two subring is
the centralizer subring $A^B$; also duality is over this base ring,
which results in left- and right-handed bialgebroids and pairings \cite{KS}.   For example, the bialgebroids of a Hopf-Galois extension are smash products of $R$ with the Hopf algebra and its dual w.r.t.\ the Galois action and the
Miyashta-Ulbrich action.  The latter is a bialgebroid with antipode,
i.e. a Hopf algebroid, applying a  theorem of Lu inspired from Poisson geometry \cite{Belg}. 

Hopf algebroids with a separable base algebra are in fact weak Hopf algebras,
which is a remarkable self-dual notion weakening the notion of Hopf algebra
sufficiently for the purposes of type $II_1$-subfactors, conformal field
theory  and other subjects \cite{NV, BNS}.  In this paper we will see another application of
weak Hopf algebras, namely an application of the main theorem in \cite{KN}
in which certain depth two Frobenius extensions enjoy weak Hopf algebra Galois action and coaction.  In section~5 we prove that a depth two Hopf subalgebra of a semisimple Hopf algebra (over an algebraically closed field
of characteristic zero) is a normal Hopf subalgebra via a map from
 Hopf algebra to weak Hopf algebra.

The paper is organized as follows.  In section~2 we set up notation for Frobenius structures on a semisimple Hopf algebra $H$ and a Hopf subalgebra $K$, a free Frobenius extension with Markov trace, separable centralizer subalgebra, symmetry conditions
studied in \cite{KN} as we prove.  As a bonus  we give a ``Frobenius-approach'' proof of Masuoka's characterization of normal Hopf subalgebra. In section~3 we develop some facts about depth two subrings,
and show that a pair of characterizations of depth two for Frobenius extension are equivalent.  In section~4 we set up notation for weak Hopf algebras, their actions and coactions as well as smash product.  We show an
equivalence of the action and coaction point of view on Galois theory
for weak Hopf algebras \cite{CDG}.  In sections~3 and~4 we update the main theorem in \cite{KN} to the Galois coaction picture of depth two in \cite{ferrara} and to weak Hopf-Galois theory in \cite{CDG}.  In section~5 the main theorem that $K$ is a normal Hopf subalgebra of $H$ is proven by defining  a nonunital algebra homomorphism from $H$ into a weak Hopf algebra $W$ via the coaction $H \to H \otimes W$
and noting that this homomorphism has kernel $HK^+ = K^+ H$ via the Galois 
isomorphism $\beta: H \otimes_K H \into H \otimes W$.

\section{Hopf subalgebra of semisimple Hopf algebra}
The ground field $k$ of all algebras in this paper is algebraically
closed of characteristic zero.  In characteristic zero,  finite dimensional algebras
enjoy equality of the classes of semisimple, separable and even strongly
separable algebras.  An algebra is strongly separable (in the sense of
Kanzaki) if the separability idempotent may be chosen symmetric w.r.t.\
the flip operator. Over an algebraically closed field, semisimple
algebras are multi-matrix algebras (in the terminology of \cite{GHJ}).
A \textit{trace} $t$ on an algebra $A$ satisfies $t(ab) = t(ba)$ for all $a,b \in A$,
is \textit{normalized} if $t(1_A) = 1$ and \textit{nondegenerate} if $t(ab) = 0$ for all $b \in A$ implies $a = 0$.    Unadorned
tensors and hom-groups are to be understood
as being over $k$.  

Let $H$ be a semisimple Hopf algebra and
$K$ a Hopf subalgebra. Some elementary facts follow from the first three chapters in \cite{M}. The natural modules
$H_K$ and ${}_KH$ are free; we let
the rank be $[H:K] = n$.  There is a nonzero two-sided integral $\Lambda_H$ in $H$ with  counit value $\eps(\Lambda_H) = \dim_k H$ and paired with an integral $t_H$  in $H^*$ such that  $t_H(1_H) = 1$. By  Larson-Radford, we have  $t_H(\Lambda_H) = 1$, the dual Hopf algebra $H^*$ is semisimple and the
square of the antipode $S^2 = \id_H$, so $t_H$ is a trace on $H$.

 The Hopf subalgebra $K$ is necessarily  semisimple.  Let $t_K \in K^*$ and $\Lambda_K \in K$ be a dual pair of two-sided integrals for $K$, satisfying $t_K(\Lambda_K) = 1$,
$\eps(\Lambda_K) = \dim_k K$ and $t_K(1_H) = 1$. We note that the restriction of $t_H$ to $K$ is $t_K$, since the dual of the inclusion
Hopf algebra monomorphism 
$\iota: K \into H$ is the restriction
epimorphism $\iota^*: H^* \rightarrow K^*$ 
which preserves integrals. Moreover,
the integral $\Lambda_H$ factors into
$\Lambda_K$ times another element $\Lambda \in H$: 
\begin{equation}
\label{eq: lamb}
\Lambda_H =  \Lambda_K \Lambda.
\end{equation}
   
The trace $t_H$ on $H$ is a nondegenerate
trace with dual bases given by $\{ S(\Lambda^H\2) \}$, $\{ \Lambda^H\1 \}$, (we set $\Lambda^H = \Lambda_H$ to sidestep a notational horror) such that for all $a \in H$,
\begin{equation}
S(\Lambda^H\2) t_H(\Lambda^H\1 a) = a
\end{equation}
and a similar left-handed equation for $\id_H$.  The trace $t_H$
induces an $H$-bimodule (symmetric Frobenius) isomorphism
$H \cong H^*$ via $a \mapsto t_H \leftharpoonup a$, and similarly a left
module isomorphism.  (Notice that the Nakayama automorphism for a general Hopf algebra, $S^2$ with modular function acting, is the identity and plays no role for semisimple Hopf algebras: see \cite{FMS, NEFE} for details of Frobenius
structure.) 

Similarly $t_K$ is a nondegenerate trace
on $K$ and has dual bases tensor $ S(\Lambda^K\2) \otimes \Lambda^K\1 $
which is a Casimir element in $K \otimes K$.  In passing we note that $S(\Lambda^K\2)\Lambda^K\1 = \eps(\Lambda_K) \neq 0$ also implies that $K$ is a separable algebra.

By a theorem of Pareigis from 1964 two symmetric algebras like $H$ and $K$ where $H_K$ is finite projective, form a Frobenius extension.  Fishman, Montgomery and Schneider  compute the (beta!) Frobenius system
in full generality \cite{FMS}, which
restricted to our situation gives us
Frobenius homomorphism $E: H \rightarrow K$, where ($a \in H$)
\begin{equation}
E(a) = t_H(a\1 \Lambda_K) a\2
\end{equation}
with dual bases tensor in $H \otimes_K H$ given
by 
\begin{equation}
S(\Lambda\2) \otimes_K \Lambda\1.
\end{equation}
Note that multiplying together components above, 
\begin{equation}
\label{eq: index}
S(\Lambda\2)\Lambda\1 = \eps(\Lambda)
= \frac{\eps(\Lambda_H)}{\eps(\Lambda_K)}
= \frac{\dim H}{\dim K} = n.
\end{equation}

\begin{exa}
\begin{rm}
Let $k = \C$, $G > N$ be a subgroup pair of finite groups,
and consider the semisimple Hopf subalgebra pair $H = \C G \supseteq
\C N = K$.  The data in our set-up is 
$\Lambda_H = \sum_{g \in G} g$, $\Lambda_K = \sum_{h \in N} h$,
then $\Lambda = \sum_{i = 1}^n g_i$ where $n = [G:N]$ and the 
$g_i$ are left coset representatives of $N$ in $G$.  The dual semisimple
Hopf algebra $H^*$ is as usual denoted by $\oplus_{g \in G} \C p_g$
where $p_g p_h = \delta_{g,h} p_g$, $\eps(p_g) = \delta_{e,g}$ and $\cop(p_g) = \sum_{x \in G} p_{gx^{-1}} \otimes p_x$.  Then $t_H = p_e$, $t_K = p_e$ (restricted) and $K$-$K$-bimodule projection $E: H \rightarrow K$
simplifies to $E(g) = \sum_{h \in N} p_e(gh)g$ with dual bases
$\{ g_i^{-1} \}_{i=1}^n$ and $\{ g_i \}_{i=1}^n$.  
\end{rm}
\end{exa}

Recall the type of Frobenius extension in  \cite{KN} which in case it additionally is
depth two has reconstruction theorem
showing it is a weak Hopf-Galois extension.
For the sake of brevity we will call
a free Frobenius extension $A \| B$ of algebras
over a field a \textit{symmetric
Markov extension} if there is Frobenius
system $E: A \rightarrow B, \{ x_i \},
\{ y_i \}$ such that
$E(1) = 1$, $\sum_i x_iy_i = 1_A(\mbox{rank}\, A_B) = \sum_i y_i x_i$, the centralizer 
$R = A^B$ is a strongly separable algebra,
$Eu = uE$ for each $u \in R$ and 
there is a normalized nondegenerate trace $t_B$ on $B$
such that $t_A = t_B \circ E$ is a nondegenerate (normalized) trace on
$A$ as well as $t_A$ restricts to a nondegenerate trace on $R$.  
This may sound like a tall order outside of type $II_1$ subfactor theory, but
we have the following example of a symmetric Markov extension.
\begin{theorem}
A semisimple Hopf algebra  pair $H \| K$
is a symmetric Markov extension. In particular,
the centralizer is $R = C_H(K)$ is a separable
algebra.  
\end{theorem}
\begin{proof}
We compute $E(1) = t_H(\Lambda_K) = t_K(\Lambda_K) = 1$.
We have seen above that the dual bases satisfy $S(\Lambda\2)\Lambda_1
= n1_A = \Lambda\1 S(\Lambda\2)$, that $t_K$ is a nondegenerate trace
on $H$ satisfying $t_K(1_H) = 1$.  By transitivity of Frobenius extension
as well as Frobenius systems (cf. \cite{KSt}), we note that
$t_K \circ E = t_H$ since each is a Frobenius homomorphism $H \rightarrow k$
with equal dual bases tensor: using eq.~(\ref{eq: lamb}), the dual bases tensor
for $t_K \circ E$ is 
$$ S(\Lambda\2)S(\Lambda^K\2) \otimes \Lambda^K\1 \Lambda\1 = 
S(\Lambda^H\2) \otimes \Lambda^H\1. $$
We note that $Eu = uE$ for every $u$ in the centralizer $R = C_H(K)$, since
($x \in K, a \in H$) 
$$ t_K(xE(ua)) = t_H(xua) = t_H(uxa) = t_H(xau) = t_K(x E(au)) $$
from which it follows from nondegeneracy of $t_K$ that $E(au) = E(ua)$
for all $a \in H$.  

We now claim that the centralizer $R$ is a semisimple algebra.  
Note that $R$ is the invariant subalgebra of the left $K$-module
algebra structure on $H$  given by the left adjoint action of $K$ on $H$:
($x \in K, a \in H$) 
\begin{equation}
x \lact a = x\1 a S(x\2)
\end{equation}
Of course, if $r \in R$, $x \lact r = x\1 S(x\2) r = \eps(x)r$,
so $R \subseteq H^K$.  Conversely, if $y \in H^K$, then for each
$x \in K$,
$$xy = x\1 y S(x\2) x\3 = y \eps(x\1)x\2 =  yx $$
whence $ y \in R$.  

Now form the smash product algebra $H \# K$ on $H \otimes K$ with
multiplication given by
\begin{equation}
(a \otimes x)(b \otimes y) = a x\1 b S(x\2) \otimes x\3 y
\end{equation}
Consider the Morita context studied by Cohen, Fishman,
Montgomery and others \cite{M} which in our case relates the centralizer $R$
and the smash product $H \# R$ via one pairing $[,]: H \otimes_R H \rightarrow H \# K$ given by $$[a,b] = a \Lambda_K b = a (\Lambda^K\1 \lact
b) \otimes \Lambda^K\2 $$
and another pairing $(,): H \otimes_{H \# K} H \rightarrow R$
given by $$(a,b) = \Lambda_K \lact (ab).$$ The pairing $(,)$ is surjective since 
 $(1_H,1_H) = (\dim K)1_H$.  

Next we observe the action of $K$ on $H$ to be inner
and the result \cite[7.3.3]{M} shows
that there is an algebra isomorphism  between the smash product and tensor product algebra, $H \# K \cong H \otimes K$, 
via the mapping $a \# x \mapsto ax\1 \otimes x\2$ with 
inverse $a \otimes x \mapsto a S(x\1) \# x\2$.  But
$H \otimes K$ is obviously a separable algebra as the tensor product
of two separable algebras. It then has trivial radical ideal $J = 0$.
By \cite[Lemma 4.3.4]{M} the idempotent $e = 1_H \# \Lambda_K$ satisfies
$R \cong Re \cong e(H \# K)e$, as part of a Morita context with surjective
trace map.  Then the centralizer $R$ has radical ideal equal
to $eJe = 0$ by \cite[Lam, ch.\ 21]{L}.  Hence $R$ is semisimple,
indeed a strongly separable, multi-matrix $k$-algebra.

Finally we claim that $t_H$ restricted to the centralizer $R$ is a 
nondegenerate trace.  
By \cite[Goodman-de la Harpe-Jones, Prop.\ 2.5.1]{GHJ} 
there is a nondegenerate trace $\phi$ on $H$ with nondegenerate restriction to $R$.  Note that $\phi = t_H(d-)$ for some unit $d \in Z(H)$, the center
of $H$, since $H$ is a symmetric algebra. Since $\phi$ is nondegenerate
when restricted to $R$, given $r \in R$, there is $r' \in R$ such
that $\phi(rr') \neq  0 $, so $t_H(drr') = t_H(rr'd) \neq 0$,
which shows that also $t_H$ restricts to a nondegenerate trace on $R$.
\end{proof}

\subsection{Normal Hopf Subalgebras}
A Hopf subalgebra $K$ is normal in a semisimple Hopf algbra $H$ if $K$ is invariant w.r.t.\ the right adjoint action:  $S(a\1)K a\2 \subseteq K$ for each $a \in H$.  Normality of $K$ may also be characterized by $K$ being invariant under the left adjoint action.  A third characterization is that $HK^+ = K^+H$ as subsets of $H$, where $K^+ = \ker \eps_K$.
Then $HK^+$ is a Hopf ideal and $H$ is
a Hopf-Galois extension of $K$ w.r.t.\
the quotient Hopf algebra $H/ HK^+$,
 denoted by $H /\!/ K$ in \cite{SB,SB2}.  
Indeed being a Hopf-Galois extension characterizes normal Hopf subalgebra  \cite[Theorem 3.1]{ferrara}.

Below we provide a Frobenius-approach
(\`a la \cite{KSt}) proof for
Masuoka's characterization of normality
in the notation we constructed above.

\begin{prop}
Suppose $H$ is a semisimple Hopf algebra.
Then a Hopf subalgebra $K$ is normal if
and only if its integral $\Lambda_K$ is central in $H$.
\end{prop}
\begin{proof}
($\Rightarrow$)  Given $a \in H$ we first
note that $a\1 \Lambda_K S(a\2) = \eps(a)\Lambda_K$.  We compute using the
Frobenius system $(t_K \in \int_{K^*}, \,
S(\Lambda^K\2), \, \Lambda^K\1)$
and the fact we noted about the traces,  $\iota^*(t_H) = t_K$:
\begin{eqnarray*}
a\1 \Lambda_K S(a\2) & = & t_K (a\1 \Lambda_K S(a\2) S(\Lambda^K\2)) \Lambda^K\1 \\
& = & t_H(\Lambda_K S(S(a\1)\Lambda^K\2 a\2)) \Lambda^K\1 \\
& = & t_K(\Lambda_K) \eps(a) \eps(\Lambda^K\2) \Lambda^K\1 = \eps(a) \Lambda_K
\end{eqnarray*}
since $\eps \circ S = \eps$ and $K$ is
normal Hopf subalgebra. 

Then $$a \Lambda_K = a\1 \Lambda_K S(a\2) a\3 = \Lambda_K \eps(a\1) a\2 = \Lambda_K a $$
hence $\Lambda_K$ is central in $H$.

($\Leftarrow$) Recall the Frobenius homomorphism $E : H \to K$ defined
by $E(a) = t_H(a\1 \Lambda_K) a\2$.
We next apply the mapping $\overline{E} = S \circ E \circ S: H \to K$,
which then satisfies $\overline{E}(a)
= a\1 t_H(\Lambda_K a\2)$,
since $S(\Lambda_K) = \Lambda_K$
and $t_H \circ S = t_H$.
 Let $a \in H$,
$x \in K$, then
\begin{eqnarray*}
\overline{E}(a\1 x S(a\2)) & = & a\1 x\1 S(a\4) t_H(\Lambda_K a\2 x\2 S(a\3)) \\
& = & a\1 x\1 S(a\4) t_H(a\2 \Lambda_K x\2 S(a\3)) \\
& = & a\1 x S(a\2) t_H(\Lambda_K) = a\1 x S(a\2)
\end{eqnarray*}
whence $a\1 K S(a\2) \subseteq K$ and $K$
is normal. 
\end{proof}   

\section{Depth Two Extensions}
     
A subring $B$ of a ring $A$ is  right depth two (D2) if the natural bimodules
${}_AA_B$ and ${}_BA_B$ satisfy the following condition:  the tensor
product bimodule $A \otimes_B A$ is $A$-$B$-bimodule isomorphic to
a direct summand of a finite direct sum of $A$ with itself.
Left depth two ring extension $A \| B$ is defined similarly:  briefly
the defining condition in suggestive symbols is given by
\begin{equation}
\label{eq: D2}
A \otimes_B A \oplus\, * \, \cong A^n
\end{equation}
as natural $B$-$A$-bimodules for some positive integer $n$.  Notice that $A \| B$ is right D2
if and only if the opposite rings $A^{\rm op} \| B^{\rm op}$ are left
D2.  

\subsection{Projectivity} 
Although a projectivity-type condition, the right D2 condition for
algebras over a ground ring does not mean that $A \otimes_B A$ is a projective
$A \otimes B^{\rm op}$-module.  However, this is so if $B$ is a separable
algebra, the situation we have in this paper:

\begin{prop}
Suppose $A \| B$ is right depth two algebra extension.  Then $A \otimes_B A$ is projective if and only if $A$ is projective, both as natural $A \otimes B^{\rm op}$-modules. If $B$ is a separable algebra, then $A$ and $A \otimes_B A$ are projective.     
\end{prop}
\begin{proof}
The first statement follows the observation that the multiplication
mapping $\mu: A \otimes_B A \to A$, defined by $\mu(x \otimes y) = xy$
is a split $A$-$B$-epimorphism, so the module $A$ enjoys the projectivity
property if $A \otimes_B A$ does. The converse follows from the eq.~(\ref{eq: D2}).  If $B$ is separable algebra, we have two successive
split epis $A \otimes B^n \to A \otimes_B B^n \cong A^n \to A \otimes_B A$,
where the first is split by $(a_1,\ldots,a_n) \mapsto (a_1e,\ldots,a_ne)$
for some separability idempotent, $e \in B \otimes B$ where $e^1 e^2 = 1$,
in a Sweedler notation suppressing a summation over simple tensors. 
It follows that $A \otimes_B A$ is $A$-$B$-bimodule isomorphic to a direct
summand of a free rank $n$ $A \otimes B^{\rm op}$-module. 
\end{proof} 

There is a converse for a ring extension $A$ over Azumaya algebra $B$,
in which case $A \| B$ is necessarily right D2 \cite[Theorem 2.3]{anchor}.   
  
\subsection{D2 Quasibases} Left depth two extension $A \| B$ is  characterized by the existence of 
$n$ mappings $\beta_i \in \End {}_BA_B$ and $n$ central elements
$t_i \in (A \otimes_B A)^B$ such that
\begin{equation}
\label{eq: d2qb}
x \otimes_B y = \sum_{i=1}^n t_i \beta_i(x)y 
\end{equation}
for all $x, y \in A$.  Given this data, a split epi $A^n \to A \otimes_B A$
is then given by $(a_1,\ldots,a_n) \mapsto \sum_i t_i a_i$,
with section $A \otimes_B A \to A^n$ given by $x \otimes_B y \mapsto
(\beta_1(x)y, \ldots, \beta_n(x)y)$.  

Note two things from the last display equation and its right depth two variant.  First, if $B$ is in the center
of $A$, and $A$ is finite projective over
$B$, then left or right D2 quasibases are
easily formed from the projective bases.
Thus, any finite projective algebra is depth two extension of its scalars.  
Second, the centralizer $R= A^B$ of $B$ in $A$ enjoys special properties
such as for every two-sided ideal $I$ of $A$ its restriction to $R$
is $A$-invariant: $A(I \cap R) = (I \cap R) A$. This equality breaks
up into reverse inclusions for left or right depth two extensions. It
is not known if there is an example of a one-sided depth two extension.
However, it is shown in \cite{CK} that quasi-Frobenius extension $A \| B$ is left D2
if and only if it is right D2.

\begin{exa}
\begin{rm}
A first example of depth two extension when looking at eq.~(\ref{eq: D2})
(and thinking of a trivial complementary bimodule $*$) 
is a Hopf-Galois extension $A \| B$. Let $H$ be the Hopf algebra,
$A$ a right $H$-comodule algebra with coinvariant subalgebra $B$. 
For then canonical or Galois mapping
$\beta: A \otimes_B A \stackrel{\cong}{\longrightarrow} A \otimes H$,
$\beta(x \otimes_B y) = xy\0 \otimes y\1$, where $x \mapsto x\0 \otimes x\1$
is the coaction of a right $H$-comodule algebra structure on $A$,
is an $A$-$A$-bimodule isomorphism of $A \otimes_B A \cong A \otimes H$ where $A \otimes H$ has right $A$-module structure given by $(a \otimes h) \cdot c = a c\0 \otimes hc\1$.  Then $\beta$ restricts to an isomorphism
of the natural $A$-$B$-bimodules $A \otimes_B A \cong A^n$ where 
$n = \dim H$. Twisting by a bijective antipode to obtain the equivalent
Galois mapping $\beta'(x \otimes_B y) = x\0 y \otimes x\1$ would
show $A$ is left D2 as well.   
\end{rm}
\end{exa}

\begin{exa}
\label{exa-nha}
\begin{rm}
A subexample is a normal Hopf subalgebra $K$ of a Hopf algebra $H$.
Here is another way to see that $H$ is depth two over $K$. Dualize the
notion of depth two algebra homomorphism to that of codepth two (coD2) 
coalgebra homomorphism via natural bicomodules and cotensor product
replacing bimodules and tensor product in eq.~(\ref{eq: D2})
(cf. \cite[section 5]{anchor}, a perfect duality in case
algebras are finite dimensional). Now let
$L$ be the quotient Hopf algebra $H // K := H/HK^+$. Then the quotient epi $\pi: H \to L$ is  by \cite[Cor.\ 5.4]{anchor} codepth two. Then $L^* \into H^*$ is D2
by duality \cite[Theorem 5.5]{anchor}. Also $L^*$ is normal by \cite[Remark 2.1]{SB}.  Again the quotient epi $H^* \to H^* // L^*$
is coD2, whence $K \cong (H^* // L^*)^* \into H = H^{**}$ is D2
by duality and \cite[Remark 2.1]{SB}. 
\end{rm}
\end{exa}

\begin{exa}
\begin{rm}
A subexample is a finite group $G$ with a normal subgroup $H$,
where $A$ and $B$ are the complex group algebras over $G$ and $H$ respectively.  We derive the right depth two property for $A \| B$
starting with a well-known induction formula:
\begin{equation}
M \uparrow^G_H \downarrow^G_H = \oplus_{i=1}^n {}^{g_i}M
\end{equation}
where $M$ is left $H$-representation space, $G = \coprod_{i=1}^n g_iH$
is an $H$-coset decomposition of $G$, and ${}^{g_i}M$ denotes the twisted module given by $g \cdot m = g_i^{-1}g g_i m$.  

Let $M = B$, then $M \uparrow^G_H = A \otimes_B B \cong A$.
From the displayed equation, we derive $M \uparrow^G_H \downarrow^G_H \uparrow^G_H  = \oplus_{i=1}^n {}^{g_i}M\uparrow^G_H $, whence
$A \otimes_B A \cong \oplus_i A \otimes_B {}^{g_i}B$ as left $A$-modules,
but also right $B$-modules by naturality.  But $A \otimes_B {}^{g_i}B \cong
A$ via $a \otimes a' \mapsto ag_i a' g_i^{-1}$ with inverse $a \mapsto
a \otimes e$. This isomorphism is obviously left
$A$-linear, but also right $B$-linear.
 Hence ${}_A A \otimes_B A_B \cong {}_A A_B^n$.
\end{rm}
\end{exa}

\subsection{Frobenius extensions}
We now show that the depth two condition above is equivalent to the
older depth two condition in \cite[Kadison-Nikshych]{KN} for Frobenius
extension.  Recall that a Frobenius extension $A \| B$ has naturally
isomorphic functors of coinduction and induction, which applied to
the regular representation given the characterization $A_B$ is finite
projective and $ A \cong \Hom (A_B, B_B)$ as natural $B$-$A$-bimodules.
A cyclic generator of $\Hom (A_B, B_B)$ is denoted by $E: A \to B$,
a ``Frobenius homomorphism'' which is in fact left $B$-linear as well
and posesses dual bases $\{ x_i \}_{i=1}^n$ and $\{ y_i \}_{i=1}^n$:
we have the following identities of $A$-central Casimir elements, 
\begin{equation}
\sum_i x_i \otimes Ey_i = \id_A = \sum_i x_iE \otimes y_i
\end{equation}
in the standard identification isomorphisms $A \otimes_B \Hom (A_B, B_B) \cong 
\End A_B $ and $ \End {}_BA \cong \Hom ({}_BA, {}_BB) \otimes_B A $.
One deduces that in fact ${}_BA$ is finite projective and ${}_AA_B  \cong
\Hom ({}_BA, {}_BB)$ is another characterization of Frobenius extension.
Frobenius homomorphisms and their dual bases are in one-to-one correspondence with units in the centralizer $R = A^B$ via $E \mapsto Er$
and $\sum_i x_i \otimes_B y_i \mapsto \sum_i x_i  \otimes_B r^{-1}y_i$ \cite{NEFE}.  

The old definition of depth two was a condition on centralizers in the
tower of a Frobenius extension. A Frobenius extension $A \| B$ has
isomorphic tensor-square and endomorphism ring, $A \otimes_B A \cong \End A_B$ via $x \otimes y \mapsto \lambda_x \circ E \circ \lambda_y$,
where $\lambda_x$ denotes left multiplication on $A$ by $x \in A$.
The inverse mapping is given by $f \mapsto \sum_i f(x_i) \otimes_B y_i$.
  Then
$e_1 = 1_A \otimes_B 1_A$ is a cyclic generator of $A_1 = A \otimes_B A$
such that $A_1 = Ae_1 A$ and $e_1 a e_1 = E(a)e_1 = e_1 E(a)$. 
The classical endomorphism ring theorem notes that $A_1$ is a Frobenius
extension of $A$, in fact with Frobenius homomorphism $E_A = \mu: A_1 \to A$
given by $E_A(xe_1y) = xy$, with dual bases $\{ x_ie_1 \}^n_{i=1}$ and
$\{ e_1 y_i \}^n_{i=1}$.  

We note that the centralizer $\End {}_BA_B \cong A_1^B$ via $\alpha \mapsto \sum_i \alpha(x_i) e_1 y_i$.  We denote $A_1^B = T$.  
We repeat the application of the endomorphism ring theorem to
obtain $A_2 = A_1 e_2 A_1$ with Frobenius homomorphism
$E_{A_1}: A_2 \to A_1$ and centralizer $S = A_2^A \cong \End {}_AA \otimes_B A_A$.
This obtains a tower of intermediate rings
$B \into A \into A_1 \into A_2$ nested
in $A_2$.  

\begin{theorem}[\cite{KS}, \cite{KK}]
Suppose $A\| B$ is a Frobenius extension with data above.
The following two conditions are equivalent:
\begin{enumerate}
\item $A$ is depth two extension of $B$
\item The Frobenius homomorphism $E_A$ has dual bases in $T = A_1^B$
and the Frobenius homomorphism $E_{A_1}$ has dual bases in $S = A_2^A$.  
\end{enumerate}
\end{theorem}
\begin{proof} ($\Rightarrow$)  Let $\beta_i$ and $t_i$ be the left
depth two quasibases introduced above.  Note that $t_i = t_i^1 e_1 t^2_i
\in A_1^B$ as is $\sum_{j=1}^n \beta_i(x_j) e_1 y_j \in A_1^B$ for
each $i = 1,\ldots, n$.  Compute these to be dual bases for $E_A$:
given $x e_1 y \in A_1$, we obtain
\begin{eqnarray*}
\sum_{i,j} t_i^1 e_1 t_i^2 E_A(\beta_i(x_j) e_1 y_j x e_1 y) & = &
\sum_{i,j} t_i E_A(\beta_i(x_j) E(y_jx) e_1 y) \\
 & = & \sum_i t_i \beta_i(x)y = xe_1 y
\end{eqnarray*}
Similarly right D2 quasibases $\gamma_k \in \End {}_BA_B$ and
$u_k \in (A \otimes_B A)^B$ yield dual bases $\sum_{k,j}
x_j e_1 \gamma_k(y_j) \otimes_A u_k^1 e_1 u_k^2 $ for $E_A$.  

To see that $E_{A_1}$ has dual bases in $S$ apply the endomorphism ring theorem for depth two extensions in \cite{LK}, which shows that $\End A_B \cong A_1$ is necessarily depth two over $A$. With depth two quasibases
for $A_1$ over $A$  and $A_2 = A_1 e_2 A_1$ defined by $E_A$-multiplication,
we iterate the argument in the last paragraph to prove that $E_{A_1}$ too
has dual bases in the centralizer $A_2^A$.

($\Leftarrow$)  Suppose $\{ c_j \}$
and $\{ d_j \}$ are dual bases in $A_1^B$ of $E_A: A_1 \to A$.  Then 
\begin{eqnarray*}
xe_1y & = & \sum_j c_j^1 e_1 c_j^2 E_A(d_j^1 e_1 d_j^2 x e_1 y) \\
& = &  \sum_j c_j^1 e_1 c_j^2 d_j^1 E(d_j^2 x) y 
\end{eqnarray*}
is a left D2 quasibase equation (cf. eq.~\ref{eq: d2qb}) for
quasibases $c_j^1 \otimes_B c_j^2 \in (A \otimes_B A)^B$ and $d_jE(d_j^2 -) \in 
\End {}_BA_B$. Hence $A \| B$ is left
D2. We similarly show from the other dual bases equation that it is right
D2 (with right D2 quasibases
$\{ d_j \}, \{ E(- c_j^1)c_j^2 \}$.
Notice that the condition on $E_{A_1}$
is redundant.  
\end{proof}

\section{Weak Hopf-Galois extensions}

The main point to the last two sections is that a semisimple
Hopf algebra $H$ with Hopf subalgebra $K$ forms a symmetric Markov
extension, and imposing the depth two condition in any one of its algebraic
characterizations, $H \| K$ is a particular case of the object of study
in the article \cite[Kadison-Nikshych]{KN}.  The main theorem in this paper (\cite[Theorem 4.6]{KN}) 
applied to $H \| K$ tells us that the centralizer $W = H_1^K$ in the Jones
tower $K \into H \into H_1$ is a semisimple weak Hopf algebra (regular and with Haar integral), that there is a $W$-module algebra structure on $H$
with $K$ as the subalgebra of invariants such that the endomorphism algebra is isomorphic to the smash product algebra, $H_1 \cong \End H_K \cong H \# W$. In this section we review smash products of module algebras with
their weak Hopf algebras and why the smash product condition on the endomorphism ring is explicitly equivalent to the existence of a Galois
isomorphism.  We will need the Galois isomorphism for the main theorem in
the next section.

Let $k$ be a field. A weak bialgebra
$W$ is  a finite dimensional $k$-algebra and $k$-coalgebra $(W, \cop, \eps)$
such that the comultiplication $\cop: W \to W \o_k W$ is linear and multiplicative,
$\cop(ab) = \cop(a) \cop(b)$, and the counit is linear
just as for bialgebras;
however, one of the change from Hopf algebra is the weakening of the axioms that $\cop$ and $\eps$  be unital, $\cop(1) \neq
1 \o 1$ and $\eps(1_W) \neq 1_k$, but must satisfy 
\begin{equation}
1\1 \o 1\2 \o 1\3 = (\cop(1) \o 1)(1 \o \cop(1)) =
(1 \o \cop(1))(\cop(1) \o 1)
\end{equation}
 and $\eps$ may not be multiplicative, $\eps(ab) \neq \eps(a) \eps(b)$ but must satisfy $(a,b,c \in W$) 
\begin{equation}
\eps(abc) = \eps(ab\1) \eps(b\2 c) = \eps(a b\2) \eps(b\1 c).
\end{equation}
There are several important projections that result from these axioms:
\begin{eqnarray}
\Pi^L(x) &:=& \eps(1\1 x) 1\2 \\
\Pi^R(x) &:=& 1\1 \eps(x 1\2) \\
\overline{\Pi}^L(x) & := & 1\1 \eps(1\2 x) \\
 \overline{\Pi}^R(x) & := & \eps(x 1\1) 1\2 \ \ \ \ \ (\forall \, x \in W)
\end{eqnarray}
We denote $W^L := \Im \Pi^L = \Im \overline{\Pi}^R$
and $W^R := \Im \Pi^R = \overline{\Pi}^L$. (These subalgebras
are separable $k$-algebras in the presence of an antipode.)

In addition to being a weak bialgebra, a weak Hopf algebra
has an antipode $S: W \to W$ satisfying the axioms
\begin{eqnarray}
S(x\1)x\2 & = & \Pi^R(x) \label{eq: bns2} \\
x\1 S(x\2) & = & \Pi^L(x) \label{eq: bns1} \\
S(x\1) x\2 S(x\3) & = & S(x) \label{eq: bnsa} \ \ \ \  (\forall \, x \in W)
\end{eqnarray}
The antipode is necessarily bijective (for finite dimensional
weak Hopf algebras), an anti-isomorphism of algebras with inverse
denoted by $\overline{S}$.  Note that  \cite[Nikshych-Vainerman]{NV}
use the notation $\eps_t = \Pi^L$ and $\eps_s = \Pi^R$.  Also the axioms of weak Hopf algebra are self-dual
and the $k$-dual algebra-coalgebra $W^*$ is shown to be a weak Hopf algebra. 

A \textit{left integral} $\ell$ in $W$ is defined
by $a \ell = \Pi^L(a) \ell$ for all $a \in W$,
is \textit{normalized} if $\Pi^L(\ell) = 1$.
The algebra $W$ is semisimple iff there is a normalized left integral
\cite[Theorem 3.13]{BNS}.  
The integral $\ell$ is \textit{nondegenerate} if the mapping $W^* \to W$,
defined by $\phi \mapsto \ell \leftharpoonup
\phi = \phi(\ell\1) \ell\2$ is a bijection. 
The antipodal concept is a \textit{right integral} $r \in W$, which satisfies $ra = r\Pi^R(a)$ for all $a \in W$, and is normalized if $\Pi^R(r) = 1$.
In a weak Hopf algebra which is a Frobenius algebra, the space of
left integrals $\mathcal{J}^L$ is a cyclic right $W^R$-module with nondegenerate cyclic
generator $\ell \in W$ \cite[Theorem 3.16]{BNS}, $\mathcal{J}^L = \ell W^R$.
Also the space of nondegenerate left integrals $\mathcal{J}^L_* =
\ell W^R_x$ where $W^R_x$ are the invertible elements in $W^R$ \cite[eq.\
(3.43)]{BNS}. 
    
A \textit{Haar integral}  $h \in W$
is a left and right normalized integral, necessarily unique if it exists
\cite[p.\ 423]{BNS}.  In this case, there is a left integral
$\lambda \in W^*$ such that $\lambda \rightharpoonup h = 1_W$ and so for
every $w \in W$, 
\begin{equation}
w = \overline{S}(h\1) \bra \lambda, wh\2 \ket, 
\end{equation}
dual bases $\{ h\2 \}$, $\{ \overline{S}(h\1 ) \}$ for the Frobenius 
homomorphism $\lambda: W \to k$. We moreover may choose $h$ to be
the cyclic generator, $\mathcal{J}^L = h W^R$ mentioned above.  

\begin{exa}    
\begin{rm}
Note from the axioms above that a Hopf algebra is automatically a weak Hopf algebra.  For a weak Hopf algebra
that is not a Hopf algebra, consider 
the groupoid algebra on $n$-objects
with one invertible arrow between each ordered pair of objects,
which is isomorphic to the $n \times n$ matrix algebra!  The weak Hopf algebra $W = M_n(k)$ has the following structure.  
Let $e_{ij}$ denote the $(i,j)$-matrix unit.  
The weak Hopf algebra structure on $M_n(k)$ has 
 counit given by $\eps(e_{ij}) = 1$, comultiplication by $\cop(e_{ij})
= $ $e_{ij} \o e_{ij}$ and antipode given by $S(e_{ij}) = e_{ji}$
for each $i,j = 1,\ldots,n$.  In this
case, $W^L = W^R$ and is equal to the diagonal matrices.  
The corresponding projections are given by
$\Pi^L(e_{ij}) = e_{ii}$ $= \overline{\Pi}^L(e_{ij})$ and $\Pi^R(e_{ij})= e_{jj}$
$= \overline{\Pi}^R(e_{ij})$. A Haar integral is given
by $h = \frac{1}{n} \sum_{i,j} e_{ij}$.  For example, if $n = 2$, the
Haar integral is 
$
h =  \left( \begin{array}{cc}
\frac{1}{2} &  \frac{1}{2} \\
\frac{1}{2} & \frac{1}{2} 
\end{array}
\right).
$
\end{rm}
\end{exa}

There are a number of important equations in the subject  (cf.\ \cite[2.8, 2.9, 2.24]{BNS}):
\begin{eqnarray}
\Pi^L & = & S \circ \overline{\Pi}^L  \label{eq: gotcha} \\
\Pi^R & = & S \circ \overline{\Pi}^R \label{eq: over2} \\
\overline{S}(a\2)a\1 & = & \overline{\Pi}^R(a) \label{eq: over1} \\
a\2 \overline{S}(a\1) &=& \overline{\Pi}^L(a) \\
a\1 \o \Pi^L(a\2) & = & 1\1 a \o 1\2  \\
\Pi^R(a\1) \o a\2 & = & 1\1 \o a 1\2  \label{eq: pi-are} \\
\Pi^R(a)b & = & b\1 \eps(ab\2)  \\
a\Pi^L(b) & = & \eps(a\1 b) a\2  \label{eq: pi-ell} \ \ \ \ \ (\forall \, a, b \in W) 
\end{eqnarray}
where e.g.\ eq.~(\ref{eq: over1}) follows from applying the inverse-antipode
to eqs.~(\ref{eq: over2}) and~(\ref{eq: bns2}). 

A $W$-module algebra $A$ is a $W$-module structure on $A$ such
that $w \cdot (ab) = (w\1 \cdot a)(w\2 \cdot b)$ 
and $w \cdot 1 = \Pi^L(w) \cdot 1$ for all $w \in W, a,b \in A$.
An algebra $B$ is a $W^*$-comodule algebra if $B$ is a comodule
via a coaction $\rho: B \to B \otimes W^*$, $\rho(b) = b\0 \otimes b\1$
such that $\rho(ab) = \rho(a)\rho(b)$ in the tensor product algebra $B \otimes W^*$ 
for $a,b \in B$ and $\rho(1) = 1\0 \otimes \Pi^L(1\1)$.
One shows as an exercise that $A$ is left $W$-module algebra if
and only if $A$ is right $W^*$-comodule algebra (via $w \cdot a =
a\0 \bra a\1, w \ket$ and dual bases).  Moreover, the
invariants $A^W = \{ b \in A \| w \cdot b = \Pi^L(w) \cdot b, \forall
w \in W \}$ form a subalgebra equal to the coinvariants $A^{\rm co \,  W^*}$
$= \{ b \in A \| b\0 \otimes b\1 = b\0 \otimes \Pi^L(b\1) \}$.
  
Denote the coinvariant subalgebra $B = A^{\rm co \,  W^*}$.  We have the
Galois mapping $\beta: A \otimes_B A \to A \otimes W^*$
given by ($x,y \in A$)
\begin{equation}
\beta(x \otimes y) = xy\0 \otimes y\1
\end{equation}
 If $\beta$ is an isomorphism, onto its image $(A \otimes W^*)\rho(1_A)$,
we say $A \| B$ is a \textit{weak Hopf-Galois extension}, or $W^*$-Galois extension \cite{CDG}. It is shown in several sources that $\beta$
surjective onto its image implies that $A_B$ is finite projective
and $\beta$ is injective; for example, see \cite[corollary 4.3]{ferrara}.  

Suppose $A$ is a $W$-module algebra.  A \textit{smash product} algebra
$A \# W$ is defined on the vector space $A \otimes_{W^L} W$, where
$W^L$ acts by multiplication from the left on $W$, while on $A$
from the right by ($w^L \in W^L$)
\begin{equation}
a \cdot w^L = \overline{S}(w^L) \cdot a = a(w^L \cdot 1).
\end{equation} 
Let $a \# w = a \otimes w^L \in A \otimes_{W^L} W$, then
$(a \# w)(b \# v) = a(w\1 \cdot b) \# w\2v$ for $a,b \in A$
and $v, w \in W$ \cite[4.2]{NV}.  The unit is $1_A \# 1_W$, e.g.,
$$(1_A \# 1_W)(a \# W) = (1\1 \cdot a) \# 1\2 w = a \cdot S(1\1) \otimes_{W^L}
1\2w = a \# w$$
from eq.~(\ref{eq: bns2}), and a half-dozen other exercises like this one
in the definition of smash product. 

The following may also be proven using Galois corings; therefore we
omit some details and concentrate on the 
main ideas in applying the main theorem \cite{KN}, formulated
in terms of Galois action, to our main theorem in the next section,
which uses the coaction picture of Galois theory for weak Hopf algebras.

\begin{theorem}
Suppose $W$ is a semisimple Hopf algebra with Haar integral $h$,
$A$ is a $W$-module algebra with invariant subalgebra $B = A^W$.
Then the Galois mapping $\beta: A \otimes_B A \to (A \otimes W^*)\rho(1_A)$
is an isomorphism if and only if the canonical mapping $
\pi: A \# W \to \End A_B$
given by $\pi(a \# w ) = \lambda_a \circ (w \cdot -)$ is an algebra isomorphism and $A_B$ is finite projective.  
\end{theorem}
\begin{proof}[Sketch of Proof]
($\Leftarrow$) We carry over to weak Hopf algebras some traditional notations and ideas from
Hopf algebras in \cite[Montgomery, chs.\ 4, 8]{M}.  Let $A$ be identified
with $A \# 1_W$ and $W$ with $1_A \# W$.  We first note that
$AhW^RA$ is a two-sided ideal in $A \# W$: it is a left ideal since
$$ (a \# w)(b \# h) = a (w\1 \cdot b) \# w\2 h = a(w\1 \cdot b)(\Pi^L(w\2) \cdot 1) \# h = a(w \cdot b) \# h$$
and a right ideal since 
\begin{eqnarray*}
(1 \# w\2)(\overline{S}(w\1) \cdot a \# 1_W) & = & w\2 \overline{S}(w\1) \cdot a \# w\3 \\
& = & \overline{\Pi}^L(w\1) \cdot a \# w\2 \\
& = & a \# \Pi^L(w\1)w\2 = a\# w 
\end{eqnarray*}
and so
\begin{eqnarray*}
(a \# h)(b \# 1_W)(c \# w) & = & (a \# h) (b c \cdot S(1\1)1\2 \# w) \\
& = & (a \# w) (bc \# w) \\
& = & (a \# h) (1 \# w\2)(\overline{S}(w\1)\cdot (bc) \# 1_W) \\
& = & (a (h\1 \cdot 1_A)  \# h\2 w\2)(\overline{S}(w\1)\cdot (bc) \# 1_W) \\
& = & (a \# \Pi^L(h\1)h\2 w\2) (\overline{S}(w\1)\cdot (bc) \# 1_W) \\
& = & (a \# h \Pi^R(w\2)(\overline{S}(w\1)\cdot (bc) \# 1_W) 
\end{eqnarray*}

Note that $bh = hb$ for $b \in B$ 
by a computation involving \cite[eqs.\ (2.25b), (2.8b)]{BNS}.  Thus the mapping $[,]: A \otimes_B A \to A\# W$
given by $x \otimes y \mapsto xhy$ is well-defined
with image $AhW^RA$.  Note that
it is surjective if and only if the Galois mapping $\beta$ is surjective:
let $\Psi: W^* \to W$ be the bijection $\Psi(g) = h \leftharpoonup g$,
then $[,] = (\id_A \otimes \Psi) \circ \beta$ since
\begin{eqnarray*}
xy\0 \otimes (t \leftharpoonup y\1) & = & x y\0 \bra y\1, h\1 \ket \otimes h\2 \\
& = & x(h\1 \cdot y) \otimes h\2 = xhy.
\end{eqnarray*}

Choose projective bases $\{ a_i \}_{i=1}^n \subset A$
and $\{ \eta_i \}_{i=1}^n \subset \Hom (A_B,B_B) \subset \End A_B$
such that $\sum_{i=1}^n a_i \eta_i = \id_A$.  Next choose $c_i \in A \# H$
such that $\pi(c_i) = \eta_i$. Then $\pi(\sum_{i=1}^n a_i c_i) = \id_A$,
hence $\sum_{i=1}^n a_i c_i = 1_A \# 1_W$.  

Note that $w  c_i = \Pi^L(w) c_i$ in $A \# W$, since for $a \in A$
we have $\eta_i(a) \in B$, so
$$ \pi(w c_i)(a) = w \cdot \eta_i(a) = \Pi^L(w) \cdot \eta_i(a) = 
\pi(\Pi^L(w) c_i). $$

We note then that $c_i$ is in the invariants of the weak Hopf module
$A \# W$ where right $W$-comodule structure is given by
$a \# w \mapsto a \# w\1 \otimes w\2$ and left $W$-module
structure by $w \cdot (a \# w') = (w\1 \cdot a ) \# w\2 w' $.  
 By the Fundamental Theorem \cite[Theorem 3.9]{BNS}
applied  to right Hopf modules over $(W^{\rm op}, \cop, \overline{S})$)
this is then isomorphic to the trivial weak Hopf module $A \otimes_{W^L} W$.
The invariants then satisfy $\mbox{Inv}\, A \# W = A \otimes_{W^L} \mathcal{J}^L$. Whence $c_i \in \mathcal{J}^LA = h W^RA$
and $1_A \# 1_W = \sum_{i=1}^N a_i c_i$ is in the ideal $A hW^R A$.
Thus the pairing $[,]$ is surjective.  It follows that
the Galois mapping is surjective, therefore injective.

($\Rightarrow$) This is based on the commutative square
$$\begin{diagram}
\Hom (A \otimes_B A_A, A_A) &&  \rTo^{\cong}&& \End A_B \\
\uTo^{\Hom(\beta,A)} && && \uTo_{\pi}    \\
 \Hom (\overline{A \otimes W^*}_A, A_A) &&  \lTo_{\phi}   && A \# W
\end{diagram}$$
with three isomorphisms, the mapping $\phi$ being given
by $\phi(a \otimes w)(x1\0 \otimes w^*1\1) = ax\0 \bra x\1 S(w^*), w \ket$.
The module structure $\overline{A \otimes W^*}_A$ is given
by $(a 1\0 \otimes w^* 1\1) \cdot b = a b\0 \otimes w^* b\1$.
An inverse mapping is given by $g \mapsto g(1_A \otimes \overline{S}(-))$
in $\mbox{Hom}_{-W^L} (W^*,A) \cong A \otimes_{W^L} W$.  The top arrow is given by
$f \mapsto f( - \otimes_B 1_A)$ with inverse $$ \alpha \longmapsto
(a \otimes_B a' \mapsto \alpha(a)a'). $$
Since $\beta$ is an isomorphism it follows that $\pi$ is an isomorphism
of endomorphism ring with smash product and $A_B$ is finite projective
\cite[cor.\ 4.3]{ferrara}.  
\end{proof}

\section{Depth two Hopf subalgebras are normal}
The following is the main theorem in this paper.

\begin{theorem}
Let $k$ be an algebraically closed field of characteristic zero.
A depth two Hopf subalgebra $K$ of a semisimple Hopf $k$-algebra
is normal.
\end{theorem}
\begin{proof}
By theorem in section~2 $H \| K$ is a symmetric Markov extension.
By hypothesis, $H \| K$ is depth two, also in the sense of depth
two Frobenius extension in \cite[Kadison-Nikshych]{KN} as shown in
section~3. It follows from \cite[Theorem 4.6]{KN} that $H \| K$ is a weak Hopf-Galois
extension as noted in section~4, since $H_K$ is finite free and $\End H_K$
is isomorphic to the smash product of $H$ with the weak Hopf algebra
$W = \End {}_KH_K$ (isomorphic to the weak Hopf algebra
denoted by $A$ in \cite{KN}).  
As a consequence, there is a right $W$-comodule algebra structure on $H$
denoted by $a \mapsto a\0 \otimes a\1$
with Galois isomorphism $\beta: H \otimes_K H \to H \otimes W$
given by $\beta(a \otimes b) = ab\0 \otimes b\1$,
with image $H1\0 \otimes W1\1$.  

Consider the following mapping of $\Phi: H \to W$, defined by
$\Phi(a) = \eps_H(a\0)a\1$.  The mapping $\Phi$ is a possibly non-unital algebra homomorphism since the counit $\eps_H: H \to k$ is algebra homomorphism and the coaction on $H$ is non-unital homomorphic,
where $(ab)\0 \otimes (ab)\1 = a\0 b\0 \otimes a\1 b\1$ but
$1\0 \otimes 1\1$ does not need to equal $1_H \otimes 1_W$ (but is idempotent). 

Let $K^+$ denote $\ker \eps_H \cap K$.  Note that  the two-sided ideal $HK^+H$ contains
$HK^+$ and $K^+H$ and is contained in $\ker \Phi$:
if $y \in K^+$, we note that $\Phi(y) = \eps_H(y1\0) 1\1 = 0$
since $y \in  H^{\rm co \, W} \cap \ker \eps_H$.

Note $\beta(\Lambda_H \otimes a) = \Lambda_H a\0 \otimes a\1 = \Lambda_H \otimes \Phi(a) \in \Lambda_H \otimes \Im \Phi$. If $H \cong K^n$ as natural
left $K$-modules (Nichols-Zoeller freeness theorem!), we
note $k\Lambda_H \otimes_K H \cong k^n$ since $k\Lambda_H \cong
k_{\eps}$ as right $K$-modules and $k \otimes_K K \cong k$. It follows
that $\dim \Im \Phi = n$. 

Now consider the Schneider canonical isomorphism 
$\overline{\beta}: H \otimes_K H \stackrel{\cong}{\longrightarrow} H \otimes (H/K^+H)$ given by formula $\overline{\beta}(a \otimes_K b) = ab\1 \otimes \overline{b\2}$, where $x \mapsto \overline{x}$ denotes the canonical
projection of element into its coset, $H \to H/ K^+H$, in the quotient
right module-coalgebra \cite{HJS, FMS,M}. (Of course 
$y - \eps(y)1 \in K^+$ for $y \in K$.) 
A similar computation ensues:
$\overline{\beta}(\Lambda_H \otimes a) = \Lambda_H \otimes \overline{a}$,
hence $\Lambda_H \otimes_K H \cong \Lambda_H \otimes H/K^+H$,
which implies that also $\dim H/K^+H = n$.

Similarly the left-handed Schneider canonical isomorphism
$\underline{\beta}: H \otimes _K H \to (H/HK^+) \otimes H$
defined by $\underline{\beta}(a \otimes b) = \overline{a\1} \otimes a\2 b$.
From this isomorphism applied to $a \otimes_K \Lambda_H$, it follows
that $\dim H/HK^+ = n$ as well.

Finally since $HK^+H \subseteq \ker \Phi$, the homomorphism
$\Phi$ induces $H/K^+H \to \Im \Phi$, an isomorphism since
$\dim \Im \Phi = n = H/K^+H$.  Similarly $\Phi$ induces
$H / HK^+ \stackrel{\cong}{\longrightarrow} \Im \Phi$.  
Both isomorphisms factor through $H / HK^+H \to \Im \Phi$,
also induced from $\Phi$ and necessarily an isomorphism.
It follows that $HK^+ = HK^+H = K^+H$, hence $K$ is normal
Hopf subalgebra of $H$.  
\end{proof} 
 
\begin{remark}
\begin{rm}
It would be interesting to obtain a proof of this theorem
or a generalization using character theory, perhaps combining
and building on the ideas in \cite{SB2} and \cite{KK}.

Another question is to ask if the proof above generalizes to
arbitrary Hopf algebras $K \subseteq H$.  The main theorem of depth two theory \cite[2.1, 5.1]{ferrara, LK} 
tells us that $H \| K$ is Galois in a more generalized sense.  
Now $H$ is a comodule algebra
with respect to an $R$-bialgebroid structure on $T = (H \otimes_K H)^K$,
where the coaction sends $H \to H \otimes_R T$.  We see that there are already problems in the definition of $\Phi$. 

The theorem and any future generalization is of potential interest to the Galois correspondence problem
for Galois extensions w.r.t.\ bialgebroids.  For example, one asks what
are the analogous results to those in field theory where normal subgroups of the Galois group correspond  to normal intermediate field extensions? 
\end{rm}
\end{remark} 
     
 By putting the main theorem together
with example~\ref{exa-nha} we have
the following characterization of normal
Hopf subalgebra:

\begin{cor}
Suppose $H$ is a semisimple Hopf algebra
over $k$.  Then a Hopf subalgebra $K$ is
depth two if and only if it is a normal
Hopf subalgebra of $H$.  
\end{cor}

Using the characterization of Galois extensions
w.r.t.\ a bialgebroid as  depth two,
balanced extensions in \cite[2.1, 5.1]{ferrara, LK}
with a note that $H_K$ free implies $H_K$
balanced module, we have a proof of 

\begin{cor}
Suppose $H$ is a semisimple Hopf $k$-algebra.  Then a Hopf subalgebra $K$ is
the coinvariant subalgebra under a Galois coaction if and only if
$K$ is a normal Hopf subalgebra of $H$.  
\end{cor}


\end{document}